\documentclass{article}

\usepackage{
 amsmath,
 amsxtra,
 amsthm,
 amssymb,
 etex,
 mathrsfs,
 mathtools,
 tikz-cd,
 xr,
 fullpage,
 enumerate,
 comment}
\usepackage[all]{xy}
\usepackage{hyperref}

\definecolor{Green}{rgb}{0.0, 0.5, 0.0}
\definecolor{blue}{rgb}{0.0, 0.0, 1.0}

\newtheorem{theorem}{Theorem}[section]
\newtheorem{lemma}[theorem]{Lemma}
\newtheorem{conjecture}[theorem]{Conjecture}
\newtheorem{proposition}[theorem]{Proposition}
\newtheorem{corollary}[theorem]{Corollary}

\theoremstyle{remark}
\newtheorem{remark}[theorem]{Remark}

\newtheorem{defn}[theorem]{Definition}

\newcommand{\QQ}{\mathbb{Q}}
\newcommand{\Qp}{\mathbb{Q}_p}

\newcommand{\ZZ}{\mathbb{Z}}
\newcommand{\Zp}{\mathbb{Z}_p}
\newcommand{\Ga}{\Gamma}

\DeclareMathOperator{\Gal}{Gal}

\DeclareMathOperator{\Sel}{Sel}
\DeclareMathOperator{\coker}{coker}

\DeclareMathOperator{\rank}{rank}
\DeclareMathOperator{\Ext}{Ext}
\DeclareMathOperator{\Ak}{Ak}

\newcommand{\vp}{\varphi}
\newcommand{\Iw}{\mathrm{Iw}}
\newcommand{\HIw}{H^1_{\Iw}}

\newcommand{\Char}{\mathrm{char}}

\newcommand{\lra}{\longrightarrow}

\newcommand{\Sss}{{\Sigma_\mathrm{ss}}}
\newcommand{\So}{{\Sigma_\mathrm{o}}}

\newcommand{\fs}{{\vec{s}}}

\newcommand{\Ep}{E[p^\infty]}

\newcommand{\fM}{\mathfrak{M}}

\newcommand{\cF}{\mathcal{F}}

\newcommand{\cyc}{\mathrm{cyc}}
\newcommand{\ps}[1]{[[ #1 ]]}

  \DeclareFontFamily{U}{wncy}{}
  \DeclareFontShape{U}{wncy}{m}{n}{<->wncyr10}{}
  \DeclareSymbolFont{mcy}{U}{wncy}{m}{n}
  \DeclareMathSymbol{\sha}{\mathord}{mcy}{"58}

\newcommand{\ilim}{\displaystyle \mathop{\varinjlim}\limits}
\newcommand{\plim}{\displaystyle \mathop{\varprojlim}\limits}

\begin{document}
\title{Akashi series and Euler characteristics of signed Selmer groups of elliptic curves with semistable reduction at primes above $p$}
 \author{ Antonio Lei\footnote{D\'epartement de Math\'ematiques et de Statistique,
Universit\'e Laval, Pavillion Alexandre-Vachon,
1045 Avenue de la M\'edecine,
Qu\'ebec, QC,
Canada G1V 0A6.
 E-mail: \texttt{antonio.lei@mat.ulaval.ca}}  \quad
  Meng Fai Lim\footnote{School of Mathematics and Statistics $\&$ Hubei Key Laboratory of Mathematical Sciences,
Central China Normal University, Wuhan, 430079, P.R.China.
 E-mail: \texttt{limmf@mail.ccnu.edu.cn}} }
\date{}
\maketitle

\begin{abstract} \footnotesize
\noindent Let $p$ be an odd prime number, and let $E$ be an elliptic curve defined over a number field $F'$ such that $E$ has semistable reduction at every prime of $F'$ above $p$ and is  supersingular at at least one prime above $p$. Under appropriate hypotheses, we compute the Akashi series of the signed Selmer groups of $E$ over a $\Zp^d$-extension over a finite extension $F$ of $F'$. As a by-product, we also compute the Euler characteristics of these Selmer groups.


\end{abstract}
\section{Introduction}
Throughout this article, $p$ will always denote an odd prime number. Let $E$ be an elliptic curve defined over a number field $F'$. If $E$ has good ordinary reduction at every prime of $F'$ above $p$ and $F$ is a finite extension of $F'$, the $p$-primary Selmer group of $E$ over the cyclotomic $\Zp$-extension $F^{\cyc}$ of $F$ is conjectured to be cotorsion over $\Zp\ps{\Ga}$ (meaning that its Pontryagin dual is torsion over $\Zp\ps{\Ga}$; see \cite{Maz}), where $\Ga= \Gal(F^{\cyc}/F)$. Granted this conjecture, Perrin-Riou \cite{PR84} and Schneider \cite{Sch83, Sch85} computed the $\Ga$-Euler characteristics of the aforementioned Selmer group and showed that its value is related to the $p$-part of the algebraic invariants
appearing in the formula of the Birch and Swinnerton-Dyer conjecture. Their calculations have since then been extended to higher dimensional $p$-adic Lie extensions (see \cite{CH,CSS,HV,Ze04,Ze09,Ze11}).

If $E$ has supersingular reduction at one prime above $p$, then the $p$-primary Selmer group of $E$ over  $F^{\cyc}$ is not expected to be  cotorsion over $\Zp\ps{\Ga}$ (see \cite{CH, Sch85}). When $F=F'=\QQ$, Kobayashi \cite{Kob} defined the  plus and minus Selmer groups of $E$ over $\QQ^\cyc$ by  constructing the plus and minus norm groups $\widehat{E}^\pm(\Qp^\cyc)$, which are subgroups of the formal group of $E$ at $p$. He was able to describe the algebraic structure of these plus and minus norm groups precisely and show that the plus and minus Selmer groups are cotorsion over $\Zp\ps{\Ga}$. These Selmer groups have been extended to different settings by various authors (see \cite{AL, BL, BL2, Kim07, KimPM, Kim14, KimPark, KO, LL, LeiS, LeiSuj, LZ, NS}). When $E$ is defined over $\QQ$ and $F$ is a number field where $p$ is unramified, Kim \cite{Kim07, KimPM} studied the structure of plus and minus Selmer groups over $F^\cyc$. In particular, he showed that these Selmer groups do not contain non-trivial submodules of finite index. This led to a formula of the $\Ga$-Euler characteristics of these Selmer groups under the assumption that the $p$-primary Selmer group of $E$ over $F$ is finite (see \cite{KimPM}). Remarkably, the Euler characteristics of the plus and minus Selmer groups turn out to be the same as the usual Selmer group in the ordinary case.

One of the key ingredients in Kim's works is a precise description of the algebraic structure of Kobayashi's plus and minus norm groups over $K^\cyc$, where $K$ is a finite unramified extension of $\Qp$. For the minus norm group, Kim's result is unconditional, whereas the plus norm group is studied under the hypothesis that $4$ does not divide $|K:\Qp|$ (see \cite{Kim07, Kim14}). In \cite{KO},  Kitajima and Otsuki were able to describe the plus norm groups even when $4$ divides $|K:\Qp|$. Furthermore, they relaxed the hypothesis that the elliptic curve $E$ is defined over $\QQ$ and  allowed $E$ to have mixed reduction types at primes above $p$. In their setting, $E$ is defined over a number field $F'$ with good reduction at all primes above $p$ and that if $u$ is a  prime of $F'$ above $p$ where $E$ has  supersingular reduction, then $F'_u=\Qp$. Let $F$ be a finite extension of $F'$ where the supersingular primes of $E$ above $p$ are unramified and let $\Sss$ denote the set of primes of $F$ lying above these supersingular primes. On choosing one of the two plus and minus norm subgroups for each prime of $\Sss$, we may define $2^{|\Sss|}$ signed Selmer groups. Such mixed signed Selmer groups were first considered by Kim in \cite{Kim14}. In \cite{AL}, the $\Ga$-Euler characteristics of these mixed signed Selmer groups have been computed. In a different vein, we may consider a $\Zp^d$-extensions of $F$, which we denote by $F_\infty$ and write $G=\Gal(F_\infty/F)$. We may define plus and minus Selmer groups of $E$ over $F_\infty$ (see \cite{Kim14, LeiSuj}). In \cite{LeiSuj}, assuming that $p$ splits completely in $F$, the $G$-Euler characteristics of the plus and minus (no mixed signs) Selmer groups have been computed.

In this article, we assume that our elliptic curve $E$ is defined over $F'$ with no additive reduction at all primes above $p$ (multiplicative reduction is allowed). 
We shall introduce certain hypotheses, labeled (S1)-(S5), in the main body of the article. Our goal is to compute the Akashi series of mixed signed Selmer groups over $F_\infty$ under hypotheses (S1)-(S4). Here, $F_\infty/F$ is a $\Zp^d$-extension and $F/F'$ is a finite extension, with $G=\Gal(F_\infty/F)$ as above.  As a by-product, we compute the $G$-Euler characteristics of these Selmer groups under the additional hypothesis (S5). Our main results are:

\begin{theorem}[{Theorem~\ref{main theorem}}]\label{thm:A}
Suppose that $(S1)-(S4)$ are satisfied.
Assume that the Pontryagin dual of a signed Selmer group of $E$ over $F^\cyc$, denoted by $X^{\fs}(E/F^{\cyc})$, is torsion over $\Zp\ps{\Ga}$. Then the Pontryagin dual of a signed Selmer group of $E$ over $F_\infty$, denoted by $X^{\fs}(E/F_{\infty})$, is torsion over $\Zp\ps{G}$, whose Akashi series is well-defined and is given by, up to a unit in $\Zp\ps\Ga$,
\[  T^r \cdot \Char_{\Ga}(X^{\fs}(E/F^{\cyc})),\]
where $r$ is the number of primes of $F^{\cyc}$ above $p$ with nontrivial decomposition group in $F_{\infty}/F^{\cyc}$ and at which $E$ has split multiplicative reduction, $T=\gamma-1$ with $\gamma$ being a topological generator of $\Ga$ and $\Char_{\Ga}(X^{\fs}(E/F^{\cyc}))$ denotes a characteristic power series of the $\Zp\ps\Ga$-module $X^{\fs}(E/F^{\cyc})$.
\end{theorem}

\begin{theorem}[{Theorem~\ref{thm:Euler}}]\label{thm:E}
Suppose that $(S1)-(S5)$ are satisfied. If  the $p$-primary Selmer group of $E$ over $F$ is finite,  then the $G$-Euler characteristic of $X^{\fs}(E/F_{\infty})$ is well-defined and is given by
\[ |\sha(E/F)[p^{\infty}]|\times\prod_{v\in \Sigma'} c_v^{(p)}\times\prod_{v\in \So}(d_v^{(p)})^2.\]
Here, $\Sigma'$ denotes the set of primes of $F$ where $E$ has bad reduction, $c_v^{(p)}$ is the highest power of $p$ dividing $|E(F_v):E_0(F_v)|$, where $E_0(F_v)$ is the subgroup of $E(F_v)$ consisting of points with nonsingular reduction modulo $v$,  $\So$ denotes the set of primes of $F$ lying above $p$ where $E$ has good ordinary reduction and $d_v^{(p)}$ is the highest power of $p$ dividing $|\tilde{E}_v(f_v)|$, where $f_v$ is the residue field of $F_v$.
\end{theorem}

The structure of our article is as follows. In \S\ref{S:alg}, we gather  preliminary algebraic results which will be used in subsequent sections of the article. In particular, we review the definition and some basic properties of Akashi series in \S\ref{Akashi section} and prove some basic results on the structure of certain Iwasawa modules in \S\ref{S:modules}. In \S\ref{elliptic curve over local}, we study the local cohomology of elliptic curves. We consider the ordinary and supersingular cases separately. In the supersingular case, we build on results of Kitajima and Otsuki for the cyclotomic $\Zp$-extension to study the {structure of the} local quotient $\displaystyle\frac{H^1(K_\infty,E[p^\infty])}{\widehat{E}^\pm(K_\infty)\otimes\Qp/\Zp}$, where $K_\infty$ is a $\Zp^2$-extension of a finite unramified extension of $\Qp$. This quotient is crucially used to define  the signed Selmer groups. We show that its Pontryagin dual is free over a two-variable Iwasawa algebra (see Corollary~\ref{H-homology of supersingular2}). This result  is one of the key ingredients in  the proof of Theorem~\ref{thm:A} and may be of independent interest. In \S\ref{Selmer}, we give the definition of mixed signed Selmer groups over the cyclotomic $\Zp$-extension, as well as a $\Zp^d$-extension. We  show how these Selmer groups can be related via Galois descent (see Lemma~\ref{Galois descent Selmer}), which also plays an important role in the proof of Theorem~\ref{thm:A}. Furthermore, we study the cotorsionness of these Selmer groups as well as a natural extension of the $\fM_H(G)$-conjecture of Coates et al. Finally, we put everything together to prove Theorems~\ref{thm:A} and \ref{thm:E} in \S\ref{proofs}. At the end of the article, we show that we may use these theorems to study the vanishing of $X^\fs(E/F_\infty)$  (see Corollaries~\ref{cor:vanish} and \ref{one for all}).

\subsection*{Acknowledgement}We would like to thank the anonymous referee for constructive comments on an earlier version of this article. The authors' research is partially supported by:  the NSERC Discovery Grants Program RGPIN-2020-04259 and RGPAS-2020-00096 (Lei) and the National Natural Science Foundation of China under Grant numbers  11550110172 and  11771164 (Lim).

\section{Preliminary algebraic results}\label{S:alg}

\subsection{Review of Akashi series} \label{Akashi section}
In this subsection, {$G$ denotes a fixed} compact pro-$p$ $p$-adic Lie group without $p$-torsion (we will mostly work with $G$ which is isomorphic to $\Zp^d$ for some integer $d$). Furthermore, we suppose that there exists  a closed normal
subgroup $H$ of $G$ such that $\Ga:= G/H\cong \Zp$.
\begin{defn}
If $M$ is a finitely generated $\Zp\ps{\Ga}$-module, {$\Char_\Gamma(M)$ denotes} a characteristic power series of $M$.
\end{defn}
Note that $\Char_\Gamma(M)$ is well-defined up to a unit in $\Zp\ps{\Ga}$. If $\Char_\Ga(M)$ is a
unit, we shall write $\Char_\Ga(M)=1$. Following  \cite{CFKSV, CSS, Ze11}, we have the following definition.
\begin{defn}
Let $M$ be  $\Zp\ps G$-module. We say that the
Akashi series of $M$ is well-defined if $H_i(H,M)$ is
$\Zp\ps{\Ga}$-torsion for every $i\ge0$. In this case, we
define $\Ak_H(M)$ to be the ($H$-)Akashi series of $M$, which is given by \[ \Ak_H(M):=\displaystyle \prod_{i\geq 0}\Char_\Ga H_i(H,M)^{(-1)^i}.\]
\end{defn}

 Note that the Akashi series is only well-defined up to a
unit in $\Zp\ps{\Ga}$. If the Akashi series of $M$ is a
unit in $\Zp\ps{\Ga}$, we shall write $\Ak_H(M)=1$. Note that since $G$ (and hence $H$) has
no $p$-torsion, $H$ has finite $p$-cohomological dimension, and
therefore, the alternating product is a finite product.

\begin{lemma} \label{Akashi exact seq}
Suppose that we are given a short exact sequence of $\Zp\ps{G}$-modules
\[ 0\lra M'\lra M\lra M''\lra 0. \]
If any two of the modules have well-defined Akashi series, so does the third. In this case,  we have
\[ \Ak_H(M) = \Ak_H(M')\Ak_H(M''). \]
\end{lemma}

\begin{proof}
  See \cite[Lemma 4.1]{CSS} or \cite[Proposition 2.2]{Ze11}.
\end{proof}

Since the group $G$ we will work with is isomorphic $\Zp^d$, the following lemma will be useful in our subsequent discussion.

\begin{lemma} \label{Akashi fg Zp}
Suppose that $G\cong H\times \Ga$ with $\dim H\geq 1$. For every $\Zp\ps{G}$-module $M$ that is finitely generated over $\Zp$, we have
\[ \Ak_H(M) = 1. \]
\end{lemma}

\begin{proof}
  See \cite[Lemma 4.5]{CSS} or \cite[Proposition 2.3]{Ze11}.
\end{proof}

We end this section by recalling a link between Akashi series and  Euler
characteristics. 

\begin{defn}
The $G$-Euler characteristics of a
$\Zp\ps{G}$-module $M$ is said to be well-defined if $H_i(G,M)$ is finite for each
$i\geq 0$. In this case, the $G$-Euler
characteristics is given by
\[ \chi(G,M) = \prod_{i\geq 0}|H_i(G,M)|^{(-1)^i}. \]
\end{defn}
Again, since $G$ has no $p$-torsion, the product in the definition of $\chi(G,M)$ is finite.

\begin{proposition} \label{Akashi Euler}
 Let $G$ be a compact $p$-adic group
without $p$-torsion, and let $H$ be a closed normal subgroup of $G$
with $G/H\cong \Zp$. Let $M$ be a finitely generated
$\Zp\ps{G}$-module whose $G$-Euler characteristics is well-defined.
Then the Akashi series of $M$ is well-defined and we have
 \[  \chi(G,M) = |\vp(\Ak_H(M))|_{p}^{-1}, \]
 where $|~|_{p}$ is the $p$-adic norm with $|p|_{p} = p^{-1}$,
 and $\vp$ is the augmentation map from $\Zp\ps{\Ga}$ to $\Zp$.
\end{proposition}

\begin{proof}
  See \cite[Theorem 3.6]{CFKSV} or \cite[Lemma 4.2]{CSS}.
\end{proof}

\subsection{Modules over two-variable Iwasawa algebras}\label{S:modules}

In this subsection, we will study modules over $\Zp\ps{G}$, where $G$ is a $p$-adic group isomorphic to $\Zp^2$. We begin by  recalling the following result on $\Zp\ps{\Ga}$-modules, where $\Ga\cong\Zp$ and we write $\Ga_n=\Ga^{p^n}$.

\begin{proposition} \label{free module}
Let $M$ be a finitely generated $\Zp\ps{\Ga}$-module and {$r$ a non-negative} integer such that $M_{\Ga_n}$ is a free $\Zp$-module of rank $rp^n$ for every $n\ge0$. Then $M$ is a free $\Zp\ps{\Ga}$-module of rank $r$.
\end{proposition}

\begin{proof}
See \cite[pp. 207, General Lemma]{Sai}.
\end{proof}

Now, fix two subgroups $H$ and $\Ga$ of $G$ so that $G\cong  H\times \Ga$ and $H\cong\Ga\cong\Zp$. For an integer $n\ge0$, we write $H_n=H^{p^n}$ and $G_n = H_n\times \Ga$ (not to be confused with $H^{p^n}\times \Ga^{p^n}$!). Let $M$ be a finitely generated $\Zp\ps{G}$-module. Since $G_n$ is a subgroup of $G$ of finite index, $M$ is also finitely generated over $\Zp\ps{G_n}$. It then follows from \cite[Lemma 4.5]{LimFine} that  $M_{H_n}$ and $H_1(H_n, M)$ are finitely generated $\Zp\ps{\Ga}$-modules. Since $H_n\cong\Zp$, we can identify $M^{H_n}$ with $H_1(H_n, M)$, and in particular, $M^{H_n}$ is finitely generated over $\Zp\ps{\Ga}$. We now come to the goal of this subsection, which is to prove the following analogue of Proposition \ref{free module}.

\begin{proposition} \label{free module2}
Let $M$ be a finitely generated $\Zp\ps{G}$-module and $r\ge0$ an integer such that  $M_{H_n}$ is a free $\Zp\ps{\Ga}$-module of rank $rp^n$ for every $n\ge0$. Then $M$ is a free $\Zp\ps{G}$-module of rank $r$.
\end{proposition}

In preparation for the proof of Proposition~\ref{free module2}, we prove the following two lemmas.

\begin{lemma} \label{stabilize submodule}
Let $M$ be a finitely generated $\Zp\ps{G}$-module. There exists an integer $n_0$ such that $M^{H_n} = M^{H_{n_0}}$ for all $n\geq n_0$.
\end{lemma}

\begin{proof}
 Since $G$ is commutative, each $M^{H_n}$ is also a $\Zp\ps{G}$-submodule of $M$. Furthermore, they form an ascending chain and hence must stabilize by the Noetherian property of $M$.
\end{proof}

\begin{lemma} \label{rank formula}
Let $M$ be a finitely generated $\Zp\ps{G}$-module. Then we have
\[ \rank_{\Zp\ps{\Ga}}\big(M_{H_n}\big) = p^n\rank_{\Zp\ps{G}}(M) + c\]
for $n\gg 0$, where $c$ is some constant independent of $n$.
\end{lemma}

\begin{proof}
It follows from \cite[Lemma 4.5]{LimFine} that
\[ \rank_{\Zp\ps{\Ga}}\big(M_{H_n}\big)= \rank_{\Zp\ps{G_n}}(M)+ \rank_{\Zp\ps{\Ga}}\big(M^{H_n}\big).\]
By Lemma \ref{stabilize submodule}, the quantity $\rank_{\Zp\ps{\Ga}}\big(M^{H_n}\big)$ stabilizes for large enough $n$. On the other hand, we have
$\rank_{\Zp\ps{G_n}}(M) = |G:G_n|\rank_{\Zp\ps{G}}(M) =  p^n\rank_{\Zp\ps{G}}(M)$. Putting these equations together, the proposition follows.
\end{proof}

We can now prove Proposition \ref{free module2}:

\begin{proof}[Proof of Proposition \ref{free module2}]
Since $M_H$ is a free $\Zp\ps{\Ga}$-module of rank $r$,  $M$ is generated by $r$ elements over $\Zp\ps{G}$ {by Nakayama's Lemma}. In other words, we have a short exact sequence
\[ 0 \lra K\lra \Zp\ps{G}^r\lra M\lra 0 \]
of $\Zp\ps{G}$-modules. 
Furthermore, it follows from the hypothesis of the proposition and Lemma \ref{rank formula} that $\rank_{\Zp\ps{G}}(M) = r$. Thus, $K$ must be torsion over $\Zp\ps{G}$. But $\Zp\ps{G}^r$ has no nontrivial torsion submodule, it follows that  $K=0$, and consequently, $M\cong\Zp\ps{G}^r$.
\end{proof}

\section{Elliptic curves over local fields} \label{elliptic curve over local}

In this section, we record certain results on elliptic curves over a $p$-adic local field. We consider the ordinary and supersingular cases separately.

\subsection{The ordinary case} \label{ordinary subsec}
Let $K$ be a finite extension of $\Qp$ and $E$ an elliptic curve defined over $K$. In this subsection, our elliptic curve $E$ is always assumed to have either good ordinary reduction or multiplicative reduction. Then from \cite[P. 150]{CG}, we have the following short exact sequence of $\Gal(\bar{K}/K)$-modules
\[ 0\lra C \lra \Ep\lra D\lra 0,\]
where $C$ and $D$ are cofree $\Zp$-modules of corank one. Furthermore, $C$ and $D$ are characterized by the fact that $C$ is divisible and that $D$ is the maximal quotient of $\Ep$ by a divisible subgroup on which  $\Gal(\bar{K}/K^{ur})$ acts  via a finite quotient. Here, $K^{ur}$ is the maximal unramified extension of $K$. In fact, as a $\Gal(\bar{K}/K)$-module, $D$ can be explicitly described as follows (see \cite{CG}):
\begin{equation}\label{eq:D}
    D = \begin{cases}  \widetilde E,& \mbox{if $E$ has good ordinary reduction}, \\
      \Qp/\Zp, & \mbox{if $E$ has split multiplicative reduction,} \\
      \Qp/\Zp\otimes\phi, & \mbox{if $E$ has nonsplit multiplicative reduction,}\end{cases}
\end{equation}
where $\widetilde E$ is the reduction of $E$ and $\phi$ is a nontrivial unramified character of $\Gal(\bar{K}/K)$.

\begin{lemma} \label{ordinary lemma}
Let $E$ be an elliptic curve defined over a finite extension $K$ of $\Qp$. Let  $K_{\infty}$ be a $\Zp^r$-extension of $K$ which contains the cyclotomic $\Zp$-extension $K^{\cyc}$. Write $H=\Gal(K_{\infty}/K^{\cyc})$. Then the following statements hold.
\begin{enumerate}
\item[$(a)$] We have
\[\frac{H^1(\mathcal{K}, \Ep)}{E(\mathcal{K})\otimes\Qp/\Zp}\cong H^1(\mathcal{K}, D)\]
for $\mathcal{K}= K^{\cyc}$ or $ K_{\infty}$.
\item[$(b)$] $H^0(K^{\cyc}, D)$ is finite if $E$ has either good ordinary or non-split multiplicative reduction. If $E$ has split multiplicative reduction, then $H^0(K^{\cyc}, D)\cong \Qp/\Zp$.
\item[$(c)$] If $r\geq 2$, then $\Ak_H(D(K_{\infty})^\vee)=1$. 
\end{enumerate}
\end{lemma}

\begin{proof}
 Assertion (a) follows from a well-known result of Coates-Greenberg \cite[Proposition 4.3]{CG}. Assertion (b) follows from the explicit description of $D$ given in \eqref{eq:D} above. Finally, assertion (c) is a consequence of Lemma \ref{Akashi fg Zp} since $\Gal(K_\infty/K)\cong H\times \Gal(K^\cyc/K)$.
\end{proof}

\subsection{The supersingular case} \label{supersingular subsection}
Let $E$ be an elliptic curve defined over $\QQ_p$ with good supersingular reduction and $a_p=1 + p - |\tilde{E}(\mathbb{F}_p)| = 0$, where $\tilde{E}$ is the reduction of $E$. Let $K$ be a finite unramified extension of $\QQ_p$. Denote by $\widehat{E}$ the formal group of $E$. For convenience, if $L$ is an extension of $\Qp$, we write $\widehat{E}(L)$ for $\widehat{E}(\mathfrak{m}_L)$, where $\mathfrak{m}_L$ is the maximal ideal of the ring of integers of $L$. Denote by $K^{\cyc}$ (resp., $K^{nr}$) the cyclotomic (resp, the unramified) $\Zp$-extension  of $K$. If $n\ge0$ is an integer, we write  $K_n$  (resp. $K^{(n)}$) for  the unique subextension of $K^{\cyc}/K$ (resp. $K^{nr}/K$) whose degree over $K$ is equal to $p^n$.

\begin{lemma} \label{supersingular points}
 The formal groups $\widehat{E}(K^{(m)}K_n)$ has no $p$-torsion for all integers $m, n\ge0$. In particular, $E(K^{(m)}K_n)$ has no $p$-torsion for every $m, n$.
 \end{lemma}

\begin{proof}
The first assertion is \cite[Proposition 3.1]{KO} or \cite[Proposition 8.7]{Kob}. For the second assertion, consider the following short exact sequence
\[ 0\lra \widehat{E}(K^{(m)}K_n)\lra E(K^{(m)}K_n) \lra \widetilde{E}(k_{m,n})\lra 0,\]
where $k_{m,n}$ is the residue field of $K^{(m)}K_n$. Since $\widetilde{E}(k_{m,n})$ has no $p$-torsion by our assumption that $E$ has good supersingular reduction, the second assertion follows from the first assertion.
\end{proof}

Following \cite{Kim07,KimPM,Kim14,KO,Kob, LeiSuj, NS}, we define
the following plus and minus norm groups.

\begin{defn} \label{plus minus defn}{We define $\widehat{E}^+(K^{(m)}K_n)$ and $\widehat{E}^-(K^{(m)}K_n)$ to be
\[  \left\{ P\in \widehat{E}(K^{(m)}K_n)~:~\mathrm{tr}_{n/l+1}(P)\in E(K^{(m)}K_{l}), 2\mid l, 0\leq l \leq n-1\right\} \]
and
\[ \left\{ P\in \widehat{E}(K^{(m)}K_n)~:~\mathrm{tr}_{n/l+1}(P)\in E(K^{(m)}K_{l}), 2\nmid l, 0\leq l \leq n-1\right\} \]
respectively,}
where $\mathrm{tr}_{n/l+1}: \widehat{E}(K^{(m)}K_n) \lra \widehat{E}(K^{(m)}K_{l+1})$ denotes the trace map with respect to the formal group law of $\widehat{E}$. 
\end{defn}
By \cite[Lemma 8.17]{Kob}, the groups $\widehat{E}^\pm(K^{(m)}K^\cyc)\otimes\Qp/\Zp$ inject into $H^1(K^{(m)}K^\cyc,\Ep)$ via the Kummer map.

In the rest of this subsection, we write $K_{\infty} = \cup_{m,n\ge0}K^{(m)}K_n$. Note that $\Gal(K_{\infty}/K)\cong \Zp^2$. Denote by $\Ga$ the Galois group $\Gal(K^{\cyc}/K)$ which is also identified with $\Gal(K^{\cyc}K^{(m)}/K^{(m)})$ for $m\ge0$. We shall also write $H=\Gal(K_{\infty}/K^{\cyc})$ which is identified with $\Gal(K^{nr}/K)$. For $m\ge0$, set $H_m = \Gal(K_{\infty}/K^{\cyc}K^{(m)})$, which is identified with $\Gal(K^{nr}/K^{(m)})$.

\begin{lemma} \label{H invariant of Knr}
We have $\Big(\widehat{E}(K^{nr})\otimes \Qp/\Zp\Big)^{H_m} =\widehat{E}(K^{(m)})\otimes \Qp/\Zp$ for $m\geq 0$.
Furthermore,
$\widehat{E}(K^{nr})\otimes \Qp/\Zp$ is a cofree $\Zp\ps{H}$-module of corank $|K:\Qp|$. In particular,
\[H^1\left(H_m, \widehat{E}(K^{nr})\otimes \Qp/\Zp\right) =0\]
for $m\geq 0$.\end{lemma} 

\begin{proof}
The first assertion is \cite[Proposition 2.10]{Kim14}. {Since $\widehat{E}(K^{(m)})\otimes \Qp/\Zp$ is cofree over $\Zp$ for each $m$,} $\Big(\widehat{E}(K^{nr})\otimes \Qp/\Zp\Big)^{H_m}$, is a cofree $\Zp$-module with $\Zp$-corank $|K :\Qp|p^m$ by Mattuck's theorem \cite{Mat}. Proposition \ref{free module} then implies that $\widehat{E}(K^{nr})\otimes \Qp/\Zp$ is a cofree $\Zp\ps{H}$-module of corank $|K:\Qp|$.
\end{proof}

We require an analog of the above lemma over $K_\infty$. As a start, we record the following.

\begin{lemma} \label{H homology of formal}
We have $H^1(H_m, \widehat{E}(K_{\infty})) =0$.
 \end{lemma}

\begin{proof}
Replacing $K$ by $K^{(m)}$, it suffices to prove the case for $K$ (or $H$).
By a well-known result of Coates-Greenberg \cite[Corollary 3.2]{CG}, we have $H^i(K^{\cyc}, \widehat{E}(\bar{K})) =0 = H^i(K_{\infty}, \widehat{E}(\bar{K}))$ for $i\geq 1$. Hence the spectral sequence
\[ H^i(H, H^j(K_{\infty}, \widehat{E}(\bar{K}))  \Longrightarrow H^{i+j}(K^{\cyc}, \widehat{E}(\bar{K}))\] degenerates yielding the required isomorphism.
\end{proof}

We can now establish the following analog of Lemma \ref{H invariant of Knr}.

\begin{proposition} \label{H invariant of Kinfty}
 We have \[ H^i\left(H_m, \widehat{E}(K_{\infty})\otimes \Qp/\Zp\right) =\begin{cases} \widehat{E}(K^{\cyc}K^{(m)})\otimes\Qp/\Zp,  & \mbox{if $i=0$}, \\
0, & \mbox{if $i\geq 1$}.\end{cases}\]
\end{proposition}

\begin{proof}
 As before, it suffices to prove the proposition for $H$.
 Since $H\cong\Zp$, the vanishing is clear for $i\geq 2$.
   By Lemma \ref{supersingular points}, we have the following short exact sequence
   \[ 0\lra \widehat{E}(K_{\infty})\lra \widehat{E}(K_{\infty})\otimes\Qp \lra \widehat{E}(K_{\infty}) \otimes \Qp/\Zp\lra 0.\]
   In view of Lemma \ref{H homology of formal}, upon taking $H$-invariant, we have
\[ 0\lra \widehat{E}(K^{\cyc})\lra \widehat{E}(K^{\cyc})\otimes\Qp \lra \Big(\widehat{E}(K_{\infty}) \otimes \Qp/\Zp\Big)^H\lra 0\]
and
\[ H^1\Big(H,\widehat{E}(K_{\infty})\otimes\Qp \Big)\cong H^1\Big(H,\widehat{E}(K_{\infty})\otimes\Qp/\Zp \Big)\]
 The isomorphism for $i=0$ follows from the short exact sequence.  Also, note that in the isomorphism above, the left-hand side is a $\Qp$-vector space, while the right-hand side is $p$-power torsion. Hence we must have
 $H^1\Big(H,\widehat{E}(K_{\infty})\otimes\Qp/\Zp \Big) = 0$.
\end{proof}

\begin{corollary} \label{H invariant of plus-minus}
 We have \[ H^i\left(H_m, \widehat{E}^\pm(K_{\infty})\otimes \Qp/\Zp\right) =\begin{cases} \widehat{E}^\pm(K^{\cyc}K^{(m)})\otimes\Qp/\Zp,  & \mbox{if $i=0$}, \\
0, & \mbox{if $i\geq 1$}.\end{cases}\]
\end{corollary}

\begin{proof}
Again, it suffices to prove the proposition for $H$. Since $H\cong\Zp$, the vanishing is clear for $i\geq 2$. By \cite[Proposition 2.6]{Kim14}, we have
\[ 0 \lra  \widehat{E}(K) \lra \widehat{E}^+(K^{\cyc})\oplus \widehat{E}^-(K^{\cyc}) \lra
\widehat{E}(K^{\cyc}) \lra 0\] and
\[ 0 \lra  \widehat{E}(K^{nr}) \lra \widehat{E}^+(K_{\infty})\oplus \widehat{E}^-(K_{\infty}) \lra
\widehat{E}(K_{\infty}) \lra 0.\]
{For simplicity, we shall write $A=\Qp/\Zp$, $\widehat{E}_L=\widehat{E}(L)$ and $\widehat{E}^\pm_L=\widehat{E}^\pm(L)$ for $L\in\{K,K^\cyc,K^{nr},K_\infty\}$. In view of Lemma \ref{supersingular points}, the exact sequences above induce the following short exact sequences}
{
\[0 \lra  \widehat{E}_K\otimes A \lra \big(\widehat{E}^+_{K^{\cyc}}\otimes A\big)\oplus \big( \widehat{E}^-_{K^{\cyc}}\otimes A \big)\lra
\widehat{E}_{K^{\cyc}}\otimes A \lra 0,\]
\[0 \lra  \widehat{E}_{K^{nr}}\otimes A \lra \big(\widehat{E}^+_{K_{\infty}}\otimes A\big)\oplus \big( \widehat{E}^-_{K_{\infty}}\otimes A \big)\lra
\widehat{E}_{K_{\infty}}\otimes A \lra 0 ,\]
which in turn fit into the following diagram
\[   \entrymodifiers={!! <0pt, .8ex>+} \SelectTips{eu}{}\xymatrixcolsep{0.75pc}\xymatrix{
    0 \ar[r]^{} & \widehat{E}_K\otimes A \ar[d] \ar[r] &  \big(\widehat{E}^+_{K^{\cyc}}\otimes A\big)\oplus \big( \widehat{E}^-_{K^{\cyc}}\otimes A \big)
    \ar[d] \ar[r] & \widehat{E}_{K^{\cyc}}\otimes A\ar[d]^{}  \ar[r] & 0\\
    0 \ar[r]^{} & \big(\widehat{E}_{K^{nr}}\otimes A\big)^H \ar[r]^{} & \big(\widehat{E}^+_{K_{\infty}}\otimes A\big)^H\oplus \big( \widehat{E}^-_{K_{\infty}}\otimes A \big)^H \ar[r] & \big(\widehat{E}_{K_{\infty}}\otimes A\big)^H.
    &  } \]}

Since the leftmost and rightmost vertical maps are isomorphisms by Lemma  \ref{H invariant of Knr} and Proposition \ref{H invariant of Kinfty}, so is the middle map. This implies the isomorphism of the corollary for $i=0$. Finally, the bottom sequence of the diagram continues in the form
{\begin{align*}
     H^1(H, \widehat{E}(K^{nr})\otimes A) \lra H^1(H,\widehat{E}^+(K_{\infty})\otimes A)\oplus H^1(H,\widehat{E}^-(K_{\infty})\otimes A) \\
    \lra H^1(H,\widehat{E}(K_{\infty})\otimes A).
\end{align*}}
Again, taking Lemma \ref{H invariant of Knr} and Proposition \ref{H invariant of Kinfty} into account,
we obtain the desired vanishing for $i=1$.
\end{proof}

We end this subsection with a discussion on the structure of  {the $\Zp[[G]]$-module} $\displaystyle\frac{H^1(K_{\infty}, \Ep)}{\widehat{E}^\pm(K_{\infty})\otimes\Qp/\Zp}$.

\begin{proposition} \label{H-homology of supersingular}
We have
\[ H^i\left(H_m, \frac{H^1(K_{\infty}, \Ep)}{\widehat{E}^\pm(K_{\infty})\otimes\Qp/\Zp}\right) =\begin{cases} \displaystyle\frac{H^1(K^{\cyc}K^{(m)}, \Ep)}{\widehat{E}^\pm(K^{\cyc}K^{(m)})\otimes\Qp/\Zp},  & \mbox{if $i=0$}, \\
0, & \mbox{if $i\geq 1$}.\end{cases}\]
\end{proposition}

\begin{proof}
 Consider the spectral sequence
  \[ H^i\big(H_m, H^j(K_{\infty}, \Ep)\big)\Rightarrow H^{i+j}(K^{\cyc}K^{(m)},\Ep). \]
  By \cite[Theorem 7.1.8(i)]{NSW}, $H^r(K_{\infty}, \Ep) = 0 =H^r(K^{\cyc}K^{(m)}, \Ep)$ for $r\geq 2$. Also, we have $H^0(K_{\infty}, \Ep) =0$ by Lemma \ref{supersingular points}. Hence the spectral sequence degenerates to yield
  \[ H^i\left(H_m, H^1(K_{\infty}, \Ep)\right) =\begin{cases} H^1(K^{\cyc}K^{(m)}, \Ep),  & \mbox{if $i=0$}, \\
0, & \mbox{if $i\geq 1$}.\end{cases}\]
The conclusion of the corollary now follows from combining the above observations with an analysis of the $H_m$-cohomology exact sequence of
\[ 0\lra \widehat{E}^\pm(K_{\infty})\otimes\Qp/\Zp \lra H^1(K_{\infty}, \Ep)\lra \frac{H^1(K_{\infty}, \Ep)}{\widehat{E}^\pm(K_{\infty})\otimes\Qp/\Zp}\lra 0\]
and taking Corollary \ref{H invariant of plus-minus} into account.
\end{proof}

\begin{corollary} \label{H-homology of supersingular2}
The module $\displaystyle\frac{H^1(K_{\infty}, \Ep)}{\widehat{E}^\pm(K_{\infty})\otimes\Qp/\Zp}$ is $\Zp\ps{G}$-cofree of corank $|K:\Qp|$.
\end{corollary} 

\begin{proof}
It follows from the preceding proposition that 
\[ \left(\displaystyle\frac{H^1(K_{\infty}, \Ep)}{\widehat{E}^\pm(K_{\infty})\otimes\Qp/\Zp}\right)^{H_m}\cong \displaystyle\frac{H^1(K^{\cyc}K^{(m)}, \Ep)}{\widehat{E}^\pm(K^{\cyc}K^{(m)})\otimes\Qp/\Zp},\]
where the latter is $\Zp\ps{\Gamma}$-cofree of corank $|K:\Qp|p^m$ by a calculation of Kitajima-Otsuki \cite[Proposition 3.32)]{KO}. The required conclusion now follows from Proposition \ref{free module2}.
\end{proof}

\section{Multiply signed Selmer groups} \label{Selmer}

Throughout this section, we fix $E$ to be an elliptic curve defined over a number field $F'$ and $F$ a finite extension of $F'$. The following assumptions will be in force.

\begin{itemize}
\item[(S1)]  There exists at least one prime $u$ of $F'$ lying above $p$ at which $E$ has good supersingular reduction.

 \item[(S2)] For each $u$ of $F'$ above $p$ at which $E$ has good supersingular reduction, we have

 (a) $F'_u\cong\Qp$ and $u$ is unramified in $F/F'$;

 (b) $a_u = 1 + p - |\tilde{E}_u(\mathbb{F}_p)| = 0$, where $\tilde{E}_u$ is the reduction of $E$ at $u$.

 \item[(S3)] $E$ does not have additive reduction at all primes of $F$ lying above $p$.
\end{itemize}

Let $F^{\cyc}$ be the cyclotomic $\Zp$-extension of $F$ and $F_n$ the intermediate subfield of $F^{\cyc}/F$ with $|F_n:F|=p^n$ for $n\ge0$.

\subsection{Selmer groups over cyclotomic $\Zp$-extensions}
From now on, $\Sigma$ is a fixed finite set of primes of $F$ which contains all the primes above $p$, all the ramified primes of $F/F'$, the bad reduction primes of $E$ and the archimedean primes. Let $F_\Sigma$ denote the maximal algebraic extension of $F$ unramified outside $\Sigma$. For any extension $\mathcal{F}$ of $F$ which is contained in $F_{\Sigma}$, we write $G_{\Sigma}(\mathcal{F})=\Gal(F_{\Sigma}/\mathcal{F})$. 
\begin{defn}
We define
\begin{align*}
\Sigma_p&=\{\text{primes of $F$ above $p$}\},   \\
\Sigma'&=\Sigma\setminus\Sigma_p,\\
\Sss&=\{u\in \Sigma_p:\text{$E$ has good supersingular reduction at $u$}\},\\
\So&=\Sigma_p\setminus\Sss.
\end{align*}

For any subset $S$ of $\Sigma$ and any extension $\mathcal{F}$ of $F$, we  write $S(\mathcal{F})$ for the set of primes of $\mathcal{F}$ above $S$.
\end{defn}

By (S2), every prime in $\Sss$ is totally ramified in $F^{\cyc}/F$. In particular, for each such prime $v$, there is a unique prime of $F_n$ lying above the said prime. We then write $\widehat{E}_v$ for the formal group $E$ over $F_v$.

\begin{defn}
Let $\fs=(s_v)_{v\in \Sss}\in\{+,-\}^{\Sss}$. The signed Selmer group $\Sel^{\fs}(E/F_n)$ is defined {to be the kernel of 
\begin{align*}
&H^1(G_\Sigma(F_n),\Ep)\lra \\
&\bigoplus_{v\in \Sss(F_n)}\frac{H^1(F_{n,v},\Ep)}{\widehat{E}^{s_v}(F_{n,v})\otimes\Qp/\Zp}\times\bigoplus_{v\in \Sigma(F_n)\setminus\Sss(F_n)}\frac{H^1(F_{n,v},\Ep)}{E(F_{n,v})\otimes\Qp/\Zp} .
\end{align*}
}
We set $\Sel^{\fs}(E/F^{\cyc}) = \ilim_n\Sel^{\fs}(E/F_n)$ and write $X^{\fs}(E/F^{\cyc})$ for its Pontryagin dual.
\end{defn}
 For our purposes, we {work with a different} description of $\Sel^{\fs}(E/F^{\cyc})$. Note that for primes outside $p$, we have $E(F_{n,v})\otimes\Qp/\Zp =0$. For primes in $\So(F^{\cyc})$, we have
\[\frac{H^1(F^{\cyc}_{v},\Ep)}{E(F^{\cyc}_v)\otimes\Qp/\Zp}\cong H^1(F^{\cyc}_v,D_v)\]
by Lemma \ref{ordinary lemma}(a), where $D_v$ is defined as in \S \ref{ordinary subsec}.

\begin{defn}\label{defn:J}
Let $\fs=(s_v)_{v\in \Sss}\in\{+,-\}^{\Sss}$. Given $v\in \Sigma(F^{\cyc})$, we define
\[
J_v^{\fs}(E/F^{\cyc}):=\begin{cases}
\frac{H^1(F^{\cyc}_v,\Ep)}{\widehat{E}^{s_v}(F^{\cyc}_v)\otimes\Qp/\Zp}&v\in\Sss(F^{\cyc}),\\
H^1(F^{\cyc}_v,D_v)&v\in \So(F^{\cyc}),\\
H^1(F^{\cyc}_v,\Ep)&v\in \Sigma'(F^{\cyc}).
\end{cases}
\]
\end{defn}

\begin{remark}\label{rk:rewriteSel}
The Selmer group $\Sel^{\fs}(E/F^{\cyc})$ sits inside the following exact sequence:
\[
 0\lra \Sel^{\fs}(E/F^{\cyc}) \lra H^1(G_{\Sigma}(F^{\cyc}),\Ep)\lra \bigoplus_{v\in \Sigma(F^{\cyc})}J_v^{\fs}(E/F^{\cyc}).\]
\end{remark}

\begin{conjecture}\label{conj:tor}
For all choices of $\fs$, the $\Zp\ps{\Gamma}$-module $X^{\fs}(E/F^{\cyc})$ is torsion.
\end{conjecture}

When $E$ has good ordinary reduction at all primes above $p$, the above conjecture is precisely Mazur's conjecture in \cite{Maz}, which is known to hold in the case when $E$ is defined over $\QQ$ and $F$ an abelian extension of $\QQ$ (see \cite{K}). For an elliptic curve over $\QQ$ with good supersingular reduction at $p$, this conjecture has been proved to be true by Kobayashi
in \cite{Kob};  see also \cite{BL} for a generalization of this conjecture for abelian varieties and \cite{BL2} for progress towards this conjecture for CM abelian varieties. We record certain consequences of Conjecture~\ref{conj:tor}, which will be utilized in subsequent discussion of the article.

\begin{proposition} \label{torsion H2}
Suppose that $(S1)-(S3)$ are valid, and that $X^{\fs}(E/F^\cyc)$ is a torsion $\Zp\ps{\Gamma}$-module. Then  the following assertions hold.
\begin{enumerate}
\item[$(a)$] $H^2(G_{\Sigma}(F^{\cyc}),\Ep)=0$.
\item[$(b)$] There is a short exact sequence
{\begin{align*}
    0\lra \Sel^{\fs}(E/F^{\cyc}) \lra& H^1(G_{\Sigma}(F^{\cyc}),\Ep)\\
   & \lra \bigoplus_{v\in \Sigma(F^{\cyc})}J_v^\fs(E/F^{\cyc})\lra 0.
\end{align*}}
 \end{enumerate}
\end{proposition}

\begin{proof}
 See \cite[Proposition 2.7]{LL} or \cite[Proposition 4.4]{LeiS}.
\end{proof}

\subsection{Selmer groups over $\Zp^d$-extensions}
Throughout this subsection, let $F_{\infty}$ be a $\Zp^d$-extension of $F$ which  satisfies the following hypothesis.
\begin{itemize}
\item[(S4)] $F^{\cyc}\subseteq F_{\infty}$, and every $v\in\Sss(F^{\cyc})$ is unramified in $F_{\infty}/F^{\cyc}$.
\end{itemize}

Write $G=\Gal(F_{\infty}/F)$, $H=\Gal(F_{\infty}/F^{\cyc})$ and $\Ga=\Gal(F^{\cyc}/F)$. Let $L_n$ be the unique subextension of $F_{\infty}/F$ such that $\Gal(L_n/F)\cong (\ZZ/p^n)^d$. Let $\fs=(s_v)_{v\in \Sss}\in\{+,-\}^{\Sss}$. For $w\in\Sss(L_n)$, we set $s_w = s_v$, where $v$ is the prime of $F$ below $w$. By (S4), $L_{n,w}$ is the compositum of a subextension of the cyclotomic $\Zp$-extension of $F_v$ and a subextension of the unramified $\Zp$-extension of $F_v$. Hence we can define $\widehat{E}^{s_w}(L_{n,w})$ as in Definition \ref{plus minus defn}.

\begin{defn}
For $\fs=(s_v)_{v\in \Sss}\in\{+,-\}^{\Sss}$, the signed Selmer group $\Sel^{\fs}(E/L_n)$ is then defined {to be the kernel of
\begin{align*}
     &H^1(G_\Sigma(L_n),\Ep)\lra\\
     &\bigoplus_{w\in \Sss(L_n)}\frac{H^1(L_{n,w},\Ep)}{\widehat{E}^{s_w}(L_{n,w})\otimes\Qp/\Zp}\times\bigoplus_{w\in \Sigma(L_n)\setminus\Sss(L_n)}\frac{H^1(L_{n,w},\Ep)}{E(L_{n,w})\otimes\Qp/\Zp} .
\end{align*}}
We set $\Sel^{\fs}(E/F_{\infty}) = \ilim_n\Sel^{\fs}(E/L_n)$ and write $X^\fs(E/F_\infty)$ for its Pontryagin dual.
\end{defn}

\begin{remark}
Analogous to Definition~\ref{defn:J} and Remark~\ref{rk:rewriteSel}, we define for $w\in \Sigma(F_\infty)$
\[
J_w^{\fs}(E/F_\infty):=\begin{cases}
\frac{H^1(F_{\infty,w},\Ep)}{\widehat{E}^{s_w}(F_{\infty,w})\otimes\Qp/\Zp}&w\in\Sss(F_{\infty}),\\
H^1(F_{\infty,w},D_w)&w\in \So(F_{\infty}),\\
H^1(F_{\infty,w},\Ep)&w\in \Sigma'(F_{\infty}).
\end{cases}
\]
.and we have the exact sequence
\[
 0\lra \Sel^{\fs}(E/F_{\infty}) \lra H^1(G_{\Sigma}(F_\infty),\Ep)\lra \bigoplus_{w\in \Sigma(F_\infty)}J_w^\fs(E/F_\infty).
 \]
\end{remark}

We may now relate the signed Selmer groups over $F_\infty$ to those over $F^\cyc$ via the following lemma.

\begin{lemma} \label{Galois descent Selmer}
 Suppose that $(S1)-(S4)$ hold. There is an injection
 \[\Sel^{\fs}(E/F^{\cyc})\lra \Sel^{\fs}(E/F_{\infty})^H\]
 with cokernel being cofinitely generated over $\Zp$.
\end{lemma}

\begin{proof}
Consider the following diagram
\[{   
\entrymodifiers={!! <0pt, .8ex>+} \SelectTips{eu}{}\xymatrixcolsep{0.75pc}\xymatrix{
    0 \ar[r]^{} &\Sel^{\fs}(E/F^{\cyc})  \ar[d]^\alpha \ar[r] &  H^1\big(G_{\Sigma}(F^{\cyc}),\Ep\big)
    \ar[d]^\beta  \ar[r] & \displaystyle \bigoplus_{v\in\Sigma(F^{\cyc})}J_v^\fs(E/F^{\cyc})\ar[d]^{\gamma=\oplus \gamma_v}  \\
    0 \ar[r]^{} & \Sel^{\fs}(E/F_\infty)^{H} \ar[r]^{} & H^1\big(G_{\Sigma}(F_\infty),\Ep\big)^{H} \ar[r]&
    \displaystyle\left(\bigoplus_{w\in\Sigma(F_\infty)}J_w^\fs(E/F_{\infty})\right)^H  } }\]
with exact rows. As seen in the proof of Proposition \ref{H-homology of supersingular}, {for all $w\in\Sss(F_{\infty})$, we have $E(F_{\infty,w})[p^{\infty}]=0$}. Hence we also have $E(F_{\infty})[p^{\infty}]=0$. Combining this observation with a Hochschild-Serre spectral sequence argument, we see that $\beta$ is an isomorphism. Thus, $\alpha$ is injective.

It remains to show that $\ker\gamma$ is cofinitely generated over $\Zp$. By Proposition \ref{H-homology of supersingular}, $\gamma_v$ is injective for $v\in\Sss(F^{\cyc})$.
For $v\in\Sigma(F^{\cyc})\setminus\Sss(F^{\cyc})$, the kernel of $\gamma_v$ is given by $H^1(H_v, D_v(F_{\infty,w}))$ or $H^1(H_v, \Ep(F_{\infty,w}))$ accordingly to $v$ divides $p$ or not, where $H_v$ is the decomposition group of $H$ with respect to a prime $w$ of $F_{\infty}$ above $v$. The conclusion now follows from the fact that the cohomology groups
$H^1(H,W)$ are cofinitely generated over $\Zp$ for any $p$-adic Lie group $H$ and any $\Zp$-cofinitely generated $H$-module $W$. ({Note: In fact, since $\Zp^d$-extension is unramified outside $p$ (cf. \cite[Theorem 1]{Iw73}), we have $H_v=1$ for $v\in\Sigma'(F_\infty)$, and so one even has $\ker\gamma_v=0$ for these primes. This latter observation will be used in the proof of Theorem \ref{main theorem}.})
\end{proof}
 
\begin{remark}
The above lemma is certainly well-known when $E$ has good ordinary reduction (see \cite{CH, CSS, HV}).
Unlike the ordinary case, where the argument is quite formal, the supersingular situation is less straightforward. This is because we do not have a nice enough descent theory for the formal group of an elliptic curve in ramified towers (see \cite[Section 6]{LZ}). In the setting considered in the present article, thanks to hypothesis (S4), we have applied results of Kim \cite{Kim14} and Kitajima-Otuski \cite{KO} to obtain such a descent theory in \S\ref{supersingular subsection}. 
\end{remark}

We now state the following natural generalisation of Conjecture \ref{conj:tor}.

\begin{conjecture}\label{conj:tor2}
For all choices of $\fs\in\{+,-\}^\Sss$, the Selmer group $X^{\fs}(E/F_{\infty})$ is torsion over $\Zp\ps{G}$.
\end{conjecture}

When $E$ has good ordinary reduction at all primes above $p$, the above conjecture is a natural extension of Mazur's conjecture (see \cite{CH, CSS, HO, HV,  OcV03}). When $E$ has supersingular reduction with $F'=\QQ$ and $F$ an imaginary quadratic field where $p$ splits, this was studied in \cite{KimPark, LeiP, LeiS}.

Conjecture~\ref{conj:tor2} has the following  consequence, which is analogous to Proposition~\ref{torsion H2} as  a consequence of  Conjecture \ref{conj:tor}.

\begin{proposition} \label{torsion2 H2}
Suppose that $(S1)-(S4)$ are valid, and that $X^{\fs}(E/F_\infty)$ is a torsion $\Zp\ps{G}$-module. Then we have the following assertions.
\begin{enumerate}
\item[$(a)$] $H^2(G_{\Sigma}(F_\infty),\Ep)=0$.
\item[$(b)$] There is a short exact sequence
{\begin{align*}
     0\lra \Sel^{\fs}(E/F^{\cyc})\lra&  H^1(G_{\Sigma}(F_\infty),\Ep)\\
     &\lra \bigoplus_{w\in \Sigma(F_\infty)} J_w^\fs(E/F_\infty) \lra 0.
\end{align*}}
\end{enumerate}
\end{proposition}

\begin{proof}
The proof is similar to that of Proposition \ref{torsion H2} with some extra technicality. 
By \cite[Proposition A.3.2]{PR00}, we have an exact sequence
\[0\lra \Sel^{\overrightarrow{s}}(E/F_{\infty})\lra H^1(G_{\Sigma}(F_{\infty}),\Ep)\lra \bigoplus_{w\in\Sigma(F_{\infty})}J_w^\fs(E/F_{\infty})  \]
\[ \lra \mathfrak{S}^{\overrightarrow{s}}(E/F_{\infty})^{\vee}\lra H^2(G_{\Sigma}(F_{\infty}),\Ep)\lra 0,\]
where $\mathfrak{S}^{\overrightarrow{s}}(E/F_{\infty})$ is a $\Zp\ps{G}$-submodule of {$$\HIw(F_{\infty}/F, T_pE): =\plim_n H^1(G_{\Sigma}(L_n),T_pE).$$} (For the precise definition of $\mathfrak{S}^{\overrightarrow{s}}(E/F_{\infty})$, we refer readers to loc. cit. For our purposes, the submodule theoretical information suffices.) The conclusion of the proposition will follow once we can show that $\mathfrak{S}^{\overrightarrow{s}}(E/F_{\infty})=0$.

A standard corank calculation (see \cite[Theorem 3.2]{OcV03}) tells us that
{\begin{align*}
    \mathrm{corank}_{\Zp\ps{G}}\big(H^1(G_{\Sigma}(F_{\infty}),\Ep)\big)- \mathrm{corank}_{\Zp\ps{G}}\big(H^2(G_{\Sigma}(F_{\infty}),\Ep)\big)\\ =|F: \QQ|.
\end{align*}} 
 For the local summands, we also have
\[\mathrm{corank}_{\Zp\ps{G}}\left(\bigoplus_{w\in \Sigma(F_{\infty})}J_w^\fs(E/F_{\infty}) \right) = [F:\QQ],\]
where in the calculations, we made use of \cite[Theorem 4.1]{OcV03} for primes in $\Sigma(F_{\infty})\setminus\Sss(F_{\infty})$ and Corollary \ref{H-homology of supersingular2} for primes in $\Sss(F_{\infty})$.
It follows from these formulas and the above exact sequence that if $\Sel^{\overrightarrow{s}}(E/F_{\infty})$ is a cotorsion $\Zp\ps{G}$-module, then $\mathfrak{S}^{\overrightarrow{s}}(E/F_{\infty})$ is a torsion $\Zp\ps{G}$-module. 

Hence the required assertion $\mathfrak{S}^{\overrightarrow{s}}(E/F_{\infty})=0$ will follow once we can show that $\HIw(F_{\infty}/F, T_pE)$ is a torsion-free $\Zp\ps{G}$-module. To see this, we
first recall the following spectral sequence of Jannsen (\cite[Theorem 1]{Jannsen})
\[ \Ext^i_{\Zp\ps{G}}\big(H^j(G_{\Sigma}(F_{\infty}),\Ep)^{\vee},\Zp\ps{G}\big) \Longrightarrow H_{\mathrm{Iw}}^{i+j}(F_{\infty}/F, T_pE).\] 
By considering the low degree terms, we have an exact sequence
{\begin{align*}
    0\lra \Ext^1_{\Zp\ps{G}}((E(F_{\infty})[p^{\infty}])^{\vee},\Zp\ps{G}\big) \lra \HIw(F_{\infty}/F, T_pE) \\
\lra \Ext^0_{\Zp\ps{G}}\big(H^1(G_{\Sigma}(F_{\infty}),\Ep)^{\vee},\Zp\ps{G}\big).
\end{align*}  }
From the proof of Lemma \ref{Galois descent Selmer}, we have $E(F_{\infty})[p^{\infty}]=0$. Hence the leftmost term vanishes, which in turn implies that $\HIw(F_{\infty}, T_pE)$ injects into an $\Ext^0$-term. But since the latter is a reflexive $\Zp\ps{G}$-module by \cite[Corollary 5.1.3]{NSW}, $\HIw(F_{\infty}/F, T_pE)$ must be torsion-free over $\Zp\ps{G}$. This completes the proof of the proposition.
\end{proof}

We now relate Conjectures \ref{conj:tor} and \ref{conj:tor2}.

\begin{proposition} \label{tor relate tor2}
Suppose that $(S1)-(S4)$ are satisfied.
Assume that $X^{\fs}(E/F^{\cyc})$ is torsion over $\Zp\ps{\Ga}$. Then  {the $\Zp\ps{G}$-module $X^{\fs}(E/F_{\infty})$ is torsion.}
\end{proposition}

\begin{proof}
 It follows from Lemma \ref{Galois descent Selmer} that $X^{\fs}(E/F_{\infty})_H$ is torsion over $\Zp\ps{\Ga}$. Since $H$ is abelian (and hence solvable), we may apply \cite[Lemma 2.6]{HO} to conclude that $X^{\fs}(E/F_{\infty})$ is torsion over $\Zp\ps{G}$.
\end{proof}

In what follows, we discuss a natural extension of the $\mathfrak{M}_H(G)$-conjecture formulated by Coates et al. in \cite{CFKSV} for the signed Selmer groups in our setting. Although we do not use it in subsequent calculations, it may be of independent interest. See also \cite{HL, LZ} on this subject.

\begin{conjecture}[$\mathfrak{M}_H(G)$-conjecture] \label{MHG}
 For all choices of $\fs$, the module $X^{\fs}(E/F_{\infty})/X^{\fs}(E/F_{\infty})[p^{\infty}]$ is finitely generated over $\Zp\ps{H}$.
\end{conjecture}

We give a partial evidence towards this conjecture (compare with \cite[Proposition 5.6]{CFKSV} and  \cite[Theorem 6.4]{CH}). Again, we remind the reader that the result is available in this context thanks to hypothesis (S4) and descent results of Kim and Kitajima-Otsuki \cite{Kim14, KO}. 

\begin{proposition} \label{mu MHG} 
Suppose that $(S1)-(S4)$ are satisfied.
Assume that $X^{\fs}(E/F^{\cyc})$ is finitely generated over $\Zp$. Then the dual Selmer group $X^{\fs}(E/F_{\infty})$ is finitely generated over $\Zp\ps{H}$. In particular, the $\mathfrak{M}_H(G)$-conjecture is valid.
\end{proposition}

\begin{proof}
 It follows from Lemma \ref{Galois descent Selmer} and the hypothesis of the proposition that $X^{\fs}(E/F_{\infty})_H$ is finitely generated over $\Zp$. The conclusion thus follows from Nakayama's Lemma.
\end{proof}

\section{Proofs of main results}\label{proofs}

This section is devoted to proving the main results of the article (Theorems~\ref{thm:A} and \ref{thm:E} in the introduction). Throughout, we retain the setting and notation of \S\ref{Selmer}.  Furthermore, we fix a choice of $\fs=(s_v)_{v\in\Sss}\in\{+,-\}^\Sss$.

For convenience, we identify $\Zp\ps{\Ga}\cong \Zp\ps{T}$ under a choice of a topological generator of $\Ga$. We first prove our result on Akashi series of signed Selmer groups (Theorem~\ref{thm:A}).
\begin{theorem} \label{main theorem}
Suppose that $(S1)-(S4)$ are satisfied.
Assume that {the $\Zp\ps{\Ga}$-module $X^{\fs}(E/F^{\cyc})$ is torsion.} Then $X^{\fs}(E/F_{\infty})$ is torsion over $\Zp\ps{G}$, whose Akashi series is well-defined and is given by
\[ \Ak_H(X^{\fs}(E/F_{\infty})) = T^r \cdot \Char_{\Ga}(X^{\fs}(E/F^{\cyc})),\]
where $r$ is the number of primes of $F^{\cyc}$ above $p$ with nontrivial decomposition group in $F_{\infty}/F^{\cyc}$ and at which $E$ has split multiplicative reduction.
\end{theorem}

\begin{proof}
The first assertion is precisely Proposition \ref{tor relate tor2}.
  Hence it follows from Propositions \ref{torsion H2} and \ref{torsion2 H2} that $H^2(G_{\Sigma}(\mathcal{F}),\Ep) = 0$ for $\mathcal{F} = F^{\cyc}, F_{\infty}$. Also, we have $H^0(K_{\infty}, \Ep) =0$ by the proof of Lemma \ref{Galois descent Selmer}. Hence the spectral sequence \[ H^i\big(H, H^j(G_\Sigma(F_{\infty}), \Ep)\big)\Rightarrow H^{i+j}(G_\Sigma(F^{\cyc}),\Ep) \]degenerates to yield
  \begin{equation}\label{eq:H1global}
       H^i\left(H, H^1(G_\Sigma(F_{\infty}), \Ep)\right) =\begin{cases} H^1(G_\Sigma(F^{\cyc}), \Ep),  & \mbox{if $i=0$}, \\
0, & \mbox{if $i\geq 1$}.\end{cases}
  \end{equation}
On the other hand, it follows from Propositions \ref{torsion H2} and \ref{torsion2 H2} that we have a short exact sequence
\begin{equation}\label{eq:SES}
    0\lra \Sel^{\fs}(E/\mathcal{F})\lra H^1(G_{\Sigma}(\mathcal{F}),\Ep)\lra \bigoplus_{u\in\Sigma(\mathcal{F})}J_u^\fs(E/\mathcal{F}) \lra 0
\end{equation}
for $\mathcal{F} = F^{\cyc}$ and $ F_{\infty}$, and where as before, we write $J_u(E/\mathcal{F})$ for the local terms. The short exact sequence \eqref{eq:SES} for $\cF=F^{\cyc}$ and the  $H$-cohomology long exact sequence associated to \eqref{eq:SES} when $\cF=F_\infty$ fit in the following commutative diagram
\[   {\entrymodifiers={!! <0pt, .8ex>+} \SelectTips{eu}{}\xymatrixcolsep{0.75pc}\xymatrix{
    0 \ar[r]^{} &\Sel^{\fs}(E/F^{\cyc})  \ar[d] \ar[r] &  H^1\big(G_{\Sigma}(F^{\cyc}),\Ep\big) 
    \ar[d] \ar[r] & \displaystyle \bigoplus_{v\in\Sigma(F^{\cyc})}J^{\fs}_v(E/F^{\cyc})\ar[d]^{g=\oplus g_v} \longrightarrow 0 \\
    0 \ar[r]^{} & \Sel^{\fs}(E/F_\infty)^{H} \ar[r]^{} & H^1\big(G_{\Sigma}(F_\infty),\Ep\big)^{H} \ar[r] &
    \displaystyle\left(\bigoplus_{w\in\Sigma(F_\infty)}J^{\fs}_w(E/F_{\infty})\right)^H &  } }\]
with exact rows.

As already seen in the proof of Lemma \ref{Galois descent Selmer}, the middle vertical map is an isomorphism. By Proposition~\ref{H-homology of supersingular}, $g_v$ is an isomorphism for $v\in\Sss(F^{\cyc})$. Since $\Zp^d$-extension is unramified outside $p$ (cf. \cite[Theorem 1]{Iw73}), it follows that $H_v =1$ for primes outside $p$, and so we also have that $g_v$ is an isomorphism for such prime. In conclusion, we have the following short exact sequence
{\begin{align*}
    0\lra \Sel^\fs(E/F^{\cyc})\lra \Sel^\fs(E/F_{\infty})^H &\lra\\ \bigoplus_{\substack{v\in\So(F^{\cyc})\\ \dim H_v \geq 1}}&H^1(H_v, D_v(F_{\infty})) \lra 0
\end{align*} }
by the snake lemma and the isomorphisms
 \[ H^i\left(H, \Sel^\fs(E/F_\infty)\right) \cong \bigoplus_{\substack{v\in\So(F^{\cyc})\\ \dim H_v \geq 1}}H^{i+1}(H_v, D_v(F_{\infty}))\]
 for $i\geq 1$, coming from the  $H$-cohomology long exact sequence associated to \eqref{eq:SES} when $\cF=F_\infty$ (thanks to the vanishing of $    H^i\left(H, H^1(G_\Sigma(F_{\infty}), \Ep)\right)$ as given by \eqref{eq:H1global}). {Via the duality 
 \[
 H_i(H, M)\cong H^i(H,M^\vee)^\vee,
 \]
  the above calculations can be translated to yield}
 \[ \Ak_H(X^\fs(E/F_{\infty})) = \Char_\Gamma(X^\fs(E/F^{\cyc}))  \cdot \prod_{\substack{v\in\So(F^{\cyc})\\ \dim H_v \geq 1}}\frac{\Char_\Gamma(D_v(F^{\cyc}_v)^{\vee})}{\Ak_{H_v}(D_v(F_{\infty,w})^{\vee})} .\]
But by Lemma \ref{ordinary lemma}, $\Ak_{H_v}(D_v(F_{\infty,w})^{\vee})=1$. Also, if $v$ is a prime of good ordinary reduction or non-split multiplicative reduction, then $D_v(F^{\cyc}_v)$ is finite and so $\Char_\Gamma(D_v(F^{\cyc}_v)^{\vee})=1$. Finally, if $v$ is a prime of split multiplicative reduction, we have $\Char_\Gamma(D_v(F^{\cyc}_v)^{\vee}) = \Char_\Gamma(\Zp) =T$. Thus, we have proven our theorem.
\end{proof}

\begin{remark}
In \cite{CSS, Ze11}, the Akashi series are computed under the validity of the $\mathfrak{M}_H(G)$-conjecture. However, as noted in \cite[pp. 284]{LimMHG}, one can perform these computations under the weaker hypothesis that the Pontryagin dual of the signed Selmer group over $F_\infty$ is a torsion $\Zp\ps G$-module. 
\end{remark}

We introduce one last hypothesis.
\begin{itemize}
\item[(S5)] \begin{enumerate}[(a)]
     \item The elliptic curve $E$ has good reduction at all primes above $p$;
    \item For our fixed choice of $\fs$, we have $4\nmid |F_v:\Qp|$ whenever $s_v = +$.
\end{enumerate}
\end{itemize}

We can now prove Theorem~\ref{thm:E}.

\begin{theorem}\label{thm:Euler}
Suppose that $(S1)-(S5)$ are satisfied. If  the $p$-primary Selmer group $\Sel(E/F)$ is finite,  then $\chi(G,X^{\fs}(E/F_{\infty}))$ is well-defined and is given by
\[ \chi(G,X^{\fs}(E/F_{\infty})) = |\sha(E/F)[p^\infty]|\times\prod_{v\in \Sigma'} c_v^{(p)}\times\prod_{v\in \So}(d_v^{(p)})^2.\]
Here, $c_v^{(p)}$ is the highest power of $p$ dividing $|E(F_v):E_0(F_v)|$, where $E_0(F_v)$ is the subgroup of $E(F_v)$ consisting of points with nonsingular reduction modulo $v$, and $d_v^{(p)}$ is the highest power of $p$ dividing $|\tilde{E}_v(f_v)|$, where $f_v$ is the residue field of $F_v$.
\end{theorem}

\begin{proof}
By \cite[Theorem 2.3]{AL}, we have that $\Sel^{\fs}(E/F^{\cyc})$ is cotorsion over $\Zp\ps{\Ga}$ and that
\[ \chi(\Gamma,X^{\fs}(E/F^\cyc)) = |\sha(E/F)[p^{\infty}]|\times\prod_{v\in\Sigma'} c_v^{(p)}\times\prod_{v\in \So}(d_v^{(p)})^2.\]
In view of Proposition \ref{Akashi Euler} and Theorem \ref{main theorem}, it remains to show that the $G$-Euler characteristics of $X^{\fs}(E/F_{\infty})$ is well-defined. Since $G\cong\Zp^d$, it is sufficient to show that  $\Sel^\fs(E/F_{\infty})^G$ is finite by \cite[part (1) of the theorem on P.3455]{Wad}. Let $J_v(E/F)$ denote the quotient $\displaystyle\frac{H^1(F_v,\Ep)}{E(F_v)\otimes\Qp/\Zp}$ for $v\in\Sigma$. Consider the following diagram
\begin{equation}\label{diagram:euler}
 {\entrymodifiers={!! <0pt, .8ex>+} \SelectTips{eu}{}\xymatrixcolsep{0.75pc}\xymatrix{
    0 \ar[r]^{} &\Sel(E/F)  \ar[d] \ar[r] &  H^1\big(G_{\Sigma}(F),\Ep\big)
    \ar[d] \ar[r] & \displaystyle \bigoplus_{v\in\Sigma}J_v(E/F)\longrightarrow0\ar[d]^{\oplus l_v}  \\
    0 \ar[r]^{} & \Sel^{\fs}(E/F_\infty)^{G} \ar[r]^{} & H^1\big(G_{\Sigma}(F_\infty),\Ep\big)^{G} \ar[r] &
    \displaystyle\left(\bigoplus_{w\in\Sigma(F_\infty)}J_w^\fs(E/F_{\infty})\right)^G  } }
    \end{equation}
with exact rows (the surjectivity of the first row follows from the finiteness of $\Sel(E/F)$; see \cite[Proposition~3.8]{KimPM}). By a similar argument to that in the proof of Lemma \ref{Galois descent Selmer}, the middle vertical map is an isomorphism. Hence it remains to show that $\ker l_v$ is finite for every $v$. For $v\in \Sigma\setminus\Sss$, this follows from \cite[Propositions 4.1 and 4.5]{Gr03}. 

Now let $v\in \Sss$ and $w\in\Sss(F_{\infty})$ a prime above $v$. Writing $G_v=\Gal(F_{\infty,w}/F_v)$ {and $A=\Qp/\Zp$}, we have the following diagram
{\[   \entrymodifiers={!! <0pt, .8ex>+} \SelectTips{eu}{}\xymatrixcolsep{0.75pc}\xymatrix{
    0 \ar[r]^{} & E(F_v)\otimes A \ar[d]_{a_v} \ar[r] &  H^1(F_v, \Ep)
    \ar[d]_{b_v} \ar[r] & \displaystyle\frac{H^1(F_v, \Ep)}{E(F_v)\otimes A}\longrightarrow 0 \ar[d]_{l_v}\\
    0 \ar[r]^{} & \big(E^{\pm}(F_{\infty,w})\otimes A\big)^{G_v} \ar[r]^{} & H^1(F_{\infty,w}, \Ep)^{G_v} \ar[r] & \displaystyle\left(\frac{H^1(F_{\infty,w}, \Ep)}{\widehat{E}^{\pm}(F_{\infty,w})\otimes A}\right)^{G_v} &} \]}
    with exact rows and that $b_v$ is an isomorphism. Consequently, $a_v$ is injective and $\ker l_v\cong \coker a_v$. Under (S5)(b),  \cite[Proposition~2.18]{Kim14} tells us that $a_v$ is an isomorphism, which  in turn implies that $l_v$ is injective. The proof of the theorem is now complete.
\end{proof}

We now prove the following vanishing criterion. When the elliptic curve has good ordinary reduction at all primes above $p$, this was established in \cite[Proposition 4.12]{HV} and \cite[Theorem 5.11]{LimMHG}. Our result shows that the analogue assertion holds even allowing supersingular reduction.

\begin{corollary}\label{cor:vanish}
Suppose that $(S1)-(S5)$ are satisfied. Assume that {the $\Zp\ps{\Ga}$-module  $X^{\fs}(E/F^{\cyc})$ is torsion}. Then $\Ak_H(X^{\fs}(E/F_{\infty})) =1$ if and only if $\Sel^{\fs}(E/F_{\infty})=0$.
\end{corollary}

\begin{proof}
 We shall freely use the notation of Theorem \ref{thm:Euler} in the proof.
The  if direction is clear. Conversely, suppose that $\Ak_H(X^{\fs}(E/F_{\infty})) =1$. By Theorem \ref{main theorem}, we have $\Char_{\Ga}(X^{\fs}(E/F^{\cyc}))=1$, which in turn implies that $\Sel^{\fs}(E/F^{\cyc})$ is finite. Following the proof of Theorem~\ref{thm:Euler}, we may show that the restriction map
\[
\Sel(E/F)\rightarrow \Sel^\fs(E/F^\cyc)^\Ga
\]
is injective with finite cokernel. Thus,  $\Sel(E/F)$ is also finite and Theorem \ref{thm:Euler} says that $\chi(G,X^{\fs}(E/F_{\infty}))$  is well-defined. 

Since  $\Ak_H(X^{\fs}(E/F_{\infty})) =1$, Proposition \ref{Akashi Euler} {gives} $\sha(E/F)[p^{\infty}]=0$, $c_v^{(p)}=1$ for every $v\in \Sigma'$ and $|\tilde{E}_v(f_v)|=1$ for every $v\in\So$. As seen in the proof of Theorem \ref{thm:Euler}, one has $\ker l_v=0$ for $v\in\Sss$ in the diagram \eqref{diagram:euler}. By \cite[Proposition 4.1]{Gr03} and that $c_v^{(p)}=1$, it follows that $\ker l_v=0$ for $v\in\Sigma'$. For $v\in \So$, the equality $|\tilde{E}_v(f_v)|=1$ implies that $D_v(F_v)=0$. Since $F_{\infty,w}$ is a pro-$p$ extension of $F_v$ for $w$ above $v$, we have $D_v(F_{\infty,w})=0$ and hence $\ker l_v=0$ for $v\in \So$. Therefore, $\ker l_v=0$ for every $v\in \Sigma$. Thus, the diagram (\ref{diagram:euler}) gives the isomorphism
\[
\Sel(E/F)\rightarrow \Sel^\fs(E/F_\infty)^G.
\]

Recall that $\Sel(E/F)$ is finite and $\sha(E/F)[p^{\infty}]=0$. This implies that $\Sel(E/F)=0$. Consequently, $\Sel^{\fs}(E/F_{\infty})^G=0$. Since $G$ is pro-$p$, this in turn yields that $\Sel^{\fs}(E/F_{\infty})=0$ as required. 
\end{proof} 

Finally, we end our article with  the following observation, which is a generalization of \cite[Corollary 2.8]{AL}.

\begin{corollary} \label{one for all}
Assume that $(S1)-(S5)$ are valid. Suppose that $\fs$ is such that $\Sel^{\fs}(E/F_{\infty}) =0$. Then $\Sel^{\vec{t}}(E/F_{\infty})=0$ for every $\vec{t}\in \{+,-\}^\Sss$ that verifies $(S5)(b)$.
\end{corollary}

\begin{proof}
Following from the proof of Theorem~\ref{thm:Euler},  the restriction map
\[
\Sel(E/F)\rightarrow \Sel^\fs(E/F_\infty)^G
\]
is injective with finite cokernel. In particular, our hypothesis implies that $\Sel(E/F)=0$ and hence is finite. By the fact that $\chi(G,X^{\vec{s}}(E/F_{\infty}))=1$ and Theorem \ref{thm:Euler}, we  have $\sha(E/F)[p^{\infty}]=0$, $c_v^{(p)}=1$ for every $v\in \Sigma'$ and $|\tilde{E}_v(f_v)|=1$ for every $v\in\So$. One may now proceed as in the proof of Corollary \ref{cor:vanish} to show that 
$\Sel^{\vec{t}}(E/F_{\infty})=0$.
\end{proof}

\end{document}